\documentclass{amsart}
\usepackage{amsfonts}
\usepackage{amsmath,amsthm}
\usepackage{amsfonts,amssymb}
\usepackage{mathrsfs}
\usepackage{enumerate}

\usepackage{multirow}
\usepackage{booktabs}

\newtheorem{thm}{Theorem}[section]

\newtheorem{lem}[thm]{Lemma}

\newtheorem{rem}[thm]{Remark}

\numberwithin{equation}{section}

\newcommand{\al}{\alpha}

\def\vz{\varepsilon}

\def\lz{\lambda}
\def\Lz{\Lambda}

\def\az{\alpha}

\def\bz{\beta}

\def\({\Bigl(}
\def \){ \Bigr)}

 \def\RR{{\mathbb R}}

\def\x{{\mathbf x}}
\def\y{{\mathbf y}}
\def\k{{\mathbf k}}

\def\va{\varepsilon}

\begin{document}
\def\RR{\mathbb{R}}
\def\Exp{\text{Exp}}
\def\FF{\mathcal{F}_\al}

\title[] {Average case tractability of multivariate approximation with Gevrey type kernels}

\author{Wanting Lu} \address{ Department of Mathematics, Tianjin Renai College, Tianjin 301636, China.}
 \email{luwanting1234@163.com}

\author{Heping Wang} \address{ School of Mathematical Sciences, Capital Normal
University,
Beijing 100048,
 China.}
\email{wanghp@cnu.edu.cn}


\begin{abstract} We consider  multivariate  approximation
problems in the average case setting with a zero mean Gaussian
measure whose covariance kernel is  a periodic Gevrey kernel.
     We
investigate various  notions of algebraic tractability and
exponential tractability, and obtain necessary and sufficient
conditions in terms of the  parameters of the problem.

\end{abstract}

\maketitle
\input amssym.def

\section{Introduction and main results}

Gevrey spaces consisting of $C^\infty$-functions were first
introduced in 1918 by M. Gevrey \cite{G} and have played an
important role in partial differential equations. A standard
reference on Gevrey spaces is Rodino's book \cite{R}. Recently,
K\"uhn and other authors in \cite{KMU, KP} among others
investigated the optimal linear approximation of the embeddings
from Gevrey spaces on the $d$-dimensional torus $\Bbb T^d$ to
$L_2(\Bbb T^d)$, and obtained preasymptotics, asymptotics, and
strong equivalences of the approximation numbers. They also
obtained necessary and sufficient conditions for different notions
of algebraic tractability or exponential tractability   of the
above Gevrey embeddings in the worst case setting. Chen and Wang
in \cite{CW2} among others obtained the similar results about
Gevrey type embeddings on the sphere and on the ball in the worst
case setting.

This paper is devoted to discussing  multivariate  approximation
problem in the average case setting with a zero mean Gaussian
measure whose covariance kernel is  the periodic Gevrey kernel.
Here, the dimension $d$ may be large or even huge. We consider
algorithms that use finitely many evaluations of arbitrary
continuous linear functionals. For a given error threshold
$\vz\in(0,1)$ the information complexity
 is defined to be the minimal number of information operations
for which the approximation error of some algorithm is at most
$\vz$.

Tractability is the study of how information complexity $n(\vz,
d)$ depends on $\vz$ and $d$. The algebraic (classical)
tractability (ALG-tractability) describes how  $n(\vz, d)$ behaves
as a function of $d$ and $\vz^{-1}$,  while the EXP-tractability
does as one of $d$ and $(1+\ln\vz^{-1})$. Recently the  study of
ALG-tractability and EXP-tractability   has attracted much
interest, and a great number of interesting results
 have been obtained (see
\cite{CW1, DKPW, DLPW, GW, IKPW, NW1, NW2, NW3,  PP, PPW, S, SW15,
X3} and the references therein).

In this paper, we investigate ALG-tractability and
EXP-tractability of multivariate approximation problem in the
average case setting with covariance kernels being the periodic
Gevrey kernels. Such problem was considered in \cite{KMU, KP} in
the worst case setting. We give necessary and sufficient
conditions for various notions of ALG-tractability and
EXP-tractability in terms of the parameters of the problem.

Although our results are proven for continuous linear functional
information classes, it follows from \cite{LW, NW3, X5} that our
results also hold for standard information classes.

This paper is organized as follows. Subsections 2.1 and 2.2  is
devoted to introduce average case and worst case approximation
problems with Gevrey kernels. In subsection 2.3 we introduce
various notions of ALG-tractability and EXP-tractability. In
subsection 2.4 we give our main results Theorems 2.2 and 2.3.
After that, in section 3 we give the proofs of Theorems 2.2 and
2.3.

 \section{Preliminaries and main results}

\subsection{Average case approximation problem with
Gevrey kernels}\

 Let $C(\Bbb T^d)$ be the space of continuous functions on the
 $d$-dimensional torus $\Bbb T^d=[0,2\pi]^d$. The  space $C(\Bbb T^d)$  is equipped with a zero-mean Gaussian measure
 $\mu_d$
 whose  covariance kernel is the Gevrey kernel given by
\begin{align}K_{d,\alpha, \beta, p}(\mathbf{x},\mathbf{y})&=\int_{C(\Bbb T^d)}f(\x)f(\y)\mu_d(df)\notag\\ &=\sum_{\mathbf{k}\in\mathbb{Z}^d}
    \exp(-2\beta |\mathbf{k}|_{p}^{\alpha}
    )\exp(i\mathbf{k}(\mathbf{x}-\mathbf{y})),\label{2.1.1}\end{align}where
    $0<\az,\bz,p<\infty$,
    $\x,\y\in \Bbb T^d,\ \x\y=\sum\limits_{i=1}^dx_iy_i,\ i^2=-1,\ |\mathbf{k}|_p=\big(\sum\limits_{j=1}^{d}|k_j|^p\big)^{\frac1
    p}$.

We consider the multivariate problem ${\rm APP}=\{{\rm
APP}_d\}_{d\in\Bbb N}$ which is defined via the
embedding operator
\begin{equation}\label{2.1} {{\rm APP}}_d: C(\Bbb T^d)\to L_{2}(\Bbb T^d),\ \ {\rm with}\ \  {\rm APP}_d\,
f=f.\end{equation}

It is well known that, in the average case setting with the
average being with respect to a zero-mean Gaussian measure,
adaptive choice of the above information evaluations do not
essentially help, see \cite{TWW}. Hence, we can restrict
 to nonadaptive algorithms. We approximate ${\rm APP}_d$ by algorithms
$A_{n,d}f$ of the form,
\begin{equation}\label{2.2} A_{n,d}f=\phi
    _{n,d}(L_1(f),L_2(f),\dots,L_n(f)),\end{equation} where $L_i, \ i=1,\dots,n$ are continuous linear functionals on $C(\Bbb T^d)$,
and $\phi _{n,d}:\;\Bbb R^n\to L_{2}(\Bbb T^d)$ are arbitrary
measurable mappings from $\Bbb R^n$ to $L_2(\Bbb T^d)$. The
average case error of an algorithm $A_{n,d}$  is defined as
$$e^{\rm avg}(A_{n,d}):=\Big(\int_{C(\Bbb T^d)}\|{\rm
    APP}_{d} f-A_{n,d}f\|^2_{2}\ \mu_d(df)\Big)^{1/2}.$$
The $n$th  minimal average case error  is defined by
$$e^{\rm avg}(n,d):=\inf_{A_{n,d}}e^{\rm avg}(A_{n,d}),$$
where the infimum is taken over all algorithms of the form
\eqref{2.2}.

For $n=0$, we use $A_{0,d}=0$. We obtain  the so-called initial
error $e^{\rm avg}(0,d)$ defined by
$$e^{\rm avg}(0,d):=\Big(\int_{C(\Bbb T^d)}\| f\|^2_{2}\ \mu_d(df)\Big)^{1/2}.$$

Let $C_{\mu_d}:(C(\Bbb T^d))^*\rightarrow C(\Bbb T^d)$ denote the
covariance operator of $\mu_d$, as defined in \cite[Appendix
B]{NW1}. Then the induced measure $\nu_d=\mu_{d}({\rm
    APP}_{d})^{-1}$ of
$\mu_d$ is a zero-mean
 Gaussian measure  on $L_2(\Bbb T^d)$ with covariance operator
 $$C_{\nu_d}:L_2(\Bbb T^d)\rightarrow L_2(\Bbb T^d),\ \ \ C_{\nu_d}={\rm
    APP}_{d}C_{\mu_d}({\rm
    APP}_{d})^*,$$where $A^*$ is the adjoint operator of a linear operator $A$.
 Then for $f\in L_2(\Bbb T^d)$ we have
$$C_{\nu_d}(f)(\x)=(2\pi)^{-d}\int_{\Bbb T^d}K_{d,\alpha, \beta, p}(\mathbf{x},\mathbf{y})
f(\mathbf{y})d\mathbf{y},\ \mathbf{x}\in \Bbb T^d.$$ It follows
from \cite[Chapter 6]{TWW} and \cite{NW1} that $e^{\rm avg}(n,d)$
are described through the eigenvalues  of the covariance operator
$C_{\nu_d}$. Let $\{(\lambda_{d,k},\eta_{d,k})\}_{k=1}^{\infty}$
be the eigenpairs of $C_{\nu_d}$ satisfying
\begin{equation*}C_{\nu_d}\,\eta_{d,k} =\lz_{d,k}\,\eta_{d,k},\  \ \text{for all }\  k\in\mathbb{N}, \end{equation*}
and $$\lz_{d,1}\ge \lz_{d,2}\ge \cdots\ge \lz_{d,k}\ge
\cdots\ge0.$$We remark that for $\k\in \Bbb Z^d$, $$
C_{\nu_d}(\exp(i\k\x))=\exp(-2\beta|\mathbf
k|_{p}^{\alpha})\exp(i\k\x),
$$and $\{\exp(i\k\x)\}_{\k\in \Bbb Z^d}$ is an eigenfunction system of $C_{\nu_d}$.  Hence, $\{\lambda_{d,k}\}_{k=1}^{\infty}$ is
just the  nonincreasing rearrangement of $\{\exp(-2\beta|\mathbf
k|_{p}^{\alpha})\}_{\k\in\mathbb{Z}^{d}}$, and hence
$\lz_{d,1}=1$.

 From \cite{TWW, NW1}
 we get   that the $n$-th minimal average case
error is
\begin{equation}\label{2.2-0}e^{\rm avg}(n,d)=\big(\sum_{k=n+1}^{\infty}\lz_{d,k}\big)^{1/2},\end{equation}
and  it is achieved by the optimal algorithm
$$A_{n,d}^*f=\sum_{k=1}^n \langle f, \eta_{d,k} \rangle_{2}\,
\eta_{d,k}.$$That is, $$e^{\rm avg}(n,d)=\big(\int_{C(\Bbb
T^d)}\|f-A_{n,d}^*f\|^2_{2}\mu_d(df)\big)^{1/2}=\big(\sum_{k=n+1}^{\infty}\lz_{d,k}\big)^{1/2}.$$
The trace of $C_{\nu_d}$ is given by
$${\rm trace}(C_{\nu_d})=(e^{\rm avg}(0,d))^2=\int_{C(\Bbb T^d)}\|f\|^2_{2}\,\mu_d(df)=
\sum_{k=1}^{\infty}\lz_{d,k}<\infty.$$

\subsection{Worst case approximation problem with
Gevrey kernels}\

In order to prove main results in the average case setting, we
need some results about the approximation problem with Gevrey
kernels in the worst case setting.

For $0<\az,\bz,p<\infty, \ d\in\Bbb N$, let $H(K_{d,\alpha, \beta,
p})$ be the reproducing kernel Hilbert space with  reproducing
kernel being the Gevery kernel $K_{d,\alpha, \beta, p}$ given by
\eqref{2.1.1}. The space $H(K_{d,\alpha, \beta, p})$ is just the
Gevrey  spaces $G^{\alpha, \beta, p}(\Bbb T^d)$ on $\Bbb T^d$
consisting of all
 $f\in C^{\infty}(\Bbb T^d)$ such that $$\|f\|_{G^{\alpha, \beta, p}}:=\Big(
 \sum_{\mathbf{k}\in\mathbb{Z}^d}\exp(2\beta |\mathbf{k}|_{p}^{\alpha})|\hat{f}
 (\mathbf{k})|^2\Big)^{\frac1 2}<+\infty,$$where
 $$\hat{f}(\mathbf{k})=(2\pi)^{-d}\int_{[0,2\pi]^{d}}f(\mathbf{x})\exp(-i\mathbf{k}\mathbf{x})d\mathbf{x}$$are
 the Fourier coefficients of $f$.
 That is,
 $$H(K_{d,\alpha, \beta,
p})=G^{\alpha, \beta, p}(\Bbb T^d)=\Big\{f\ \big|\
\|f\|_{G^{\alpha, \beta, p}}<+\infty\Big\}.$$The  Gevrey space has
been widely applied to differential equations and approximation
theory (see \cite{CW2, G,  KMU, KP, R}).

 We consider the multivariate approximation problem
  $S=\{S_d\}_{d\in\Bbb N}$  which is defined via the
embedding operator
$$S_{d}:G^{\alpha,\beta,p}(\Bbb T^{d})\rightarrow L_{2}(\Bbb T^{d})\
 \text{with}\  S_d f = f.$$
The worst case error of an algorithm $A_{n,d}$ of the form
\eqref{2.2} is defined as
$$e^{\text{wor}}(A_{n,d})=\sup_{\|f\|_{G^{\alpha,\beta,p}}\le1}
\|S_d f-A_{n,d} f\|_2.$$ The $n$-th minimal worst
case error is defined by
$$e^{\text{wor}}(n,d)=\inf\limits_{A_{n,d}}e^{\text{wor}}(A_{n,d}).$$
For $n=0$ we set $A_{0,d}=0$. The worst case error of $A_{0,d}$ is
called the initial error and is given by
$$e^{\text{wor}}(0,d)=\sup_{\|f\|_{G^{\alpha,\beta,p}}\le1}
\|S_d f\|_{2}=\|S_d\|.$$ Let
$\lambda_{d,k}, \ k\in \mathbb{N}$ be the eigenvalues of the
operator $(S_d )^*S_d $
satisfying
$$\lz_{d,1} \ge \lz_{d,2} \ge \ldots \ge \lz_{d,k} \ge \ldots >
0.$$We remark that $\{\lambda_{d,k}\}_{k=1}^{\infty}$ is just the
nonincreasing rearrangement of $\{\exp(-2\beta|\mathbf
k|_{p}^{\alpha})\}_{\k\in\mathbb{Z}^{d}}$. The $n$-th minimal worst
case error $e^{\text{wor}}(n, d)$ is
$$e^{\text{wor}}(n,d)=\sqrt{\lz_{d,n+1}}.$$
Specifically, the initial error is given by
$$e^{\text{wor}}(0,d)=\sqrt{\lz_{d,1}}=1.$$

 The  information complexity  can be studied using either
the absolute error criterion (ABS) or the normalized error
criterion (NOR). In the average and worst case settings for
$\star\in\{{\rm
    ABS,\,NOR}\}$ and  $X\in\{\rm {wor},\,\rm{avg}\}$, we
define the information complexity $n^{X,\,\star}(\va ,d)$  as
\begin{equation} n^{ X,\, \star}(\va
    ,d):=\inf\{n \ \big|\  e^{\rm
        X}(n,d)\le \vz\, {\rm CRI}_d^X\}, \label{2.1.10}\end{equation} where
\begin{equation*}
    {\rm CRI}_d^X:=\left\{\begin{split}
        &\ \ 1, \qquad\, \quad\quad\text{for $\star$=ABS,} \\
        &e^{X}(0,d), \quad\ \, \text{ for $\star$=NOR}.
    \end{split}\right.
\end{equation*}
Note that $e^{\rm wor}(0,d)=1$. This means that the normalized
error criterion and the absolute error criterion in the worst case
setting coincide. We write $n^{\star}(\va
    ,d)$, and $CRI_d$ instead of $n^{{\rm avg},\,\star}(\va
    ,d)$, and  $CRI_d^{\rm avg}$ respectively,  for brevity. We have
    \begin{equation}n^{\rm NOR}(\va
    ,d)\le n^{\rm ABS}(\va
    ,d).\label{2.33}\end{equation}

\subsection{Notions of ALG-tractability and EXP-tractability}\

 Various notions of ALG-tractability and EXP-tractability have been discussed for
multivariate problems. In this subsection we recall the following
basic tractability notions.

Let ${\rm APP}= \{{\rm
    APP}_d\}_{d\in\Bbb N}$, $S= \{S_d\}_{d\in\Bbb N}$, $X\in\{\rm {wor},\,\rm{avg}\}$, and
$\star\in\{{\rm ABS,\,NOR}\}$.
In the average or worst  case setting for error criterion $\star$,
we say that ${\rm APP}$ or $S$ is

$\bullet$ Algebraic strongly polynomially tractable (ALG-SPT) if
there exist $ C>0$ and non-negative number $p$ such that
\begin{equation*}n^{X,\,\star}(\va,d)\leq C\varepsilon^{-p},\
    \text{for all}\ \varepsilon\in(0,1);\end{equation*}

$\bullet$ Algebraic polynomially tractable (ALG-PT)  if there
exist $ C>0$ and non-negative numbers $p,q$ such that
$$n^{X,\,\star}(\va,d)\leq Cd^{q}\varepsilon^{-p},\ \text{for all}\
d\in\mathbb{N},\ \varepsilon\in(0,1);$$

$\bullet$ Algebraic quasi-polynomially tractable (ALG-QPT) if
there exist $ C>0$ and non-negative number $t$ such that
\begin{equation*}n^{X,\,\star}(\va,d)\leq C \exp(t(1+\ln{d})(1+\ln{\varepsilon^{-1}})),\
    \text{for all}\ d\in\mathbb{N},\
    \varepsilon\in(0,1);\end{equation*}

$\bullet$ Algebraic uniformly weakly tractable (ALG-UWT)  if
$$\lim_{\varepsilon^{-1}+d\rightarrow\infty}\frac{\ln n^{X,\,\star}(\va,d)}{\varepsilon^{-\sigma}+d^{\tau}}=0,\ \text{for all}\
\sigma, \tau>0;$$

$\bullet$ Algebraic weakly tractable (ALG-WT) if
$$\lim_{\varepsilon^{-1}+d\rightarrow\infty}\frac{\ln n^{X,\,\star}(\va,d)}{\varepsilon^{-1}+d}=0;$$

$\bullet$ Algebraic $(s,t)$-weakly tractable (ALG-$(s,t)$-WT) for
fixed $s, t>0$ if
$$\lim_{\varepsilon^{-1}+d\rightarrow\infty}\frac{\ln n^{X,\,\star}(\va,d)}{\varepsilon^{-s}+d^{t}}=0.$$

Clearly, ALG-$(1,1)$-WT is the same as ALG-WT. If the multivariate approximation problem is
not ALG-WT, then this problem  is called  intractable. For
$0<s_1<s, \ 0<t_1<t$, ALG-$(s_1,t_1)$-WT $\Longrightarrow$
ALG-$(s,t)$-WT. We also have for $s,t>0$,

\vskip 2mm
\begin{center}ALG-SPT $\Longrightarrow$ ALG-PT  $\Longrightarrow$  ALG-QPT
$\Longrightarrow$ ALG-UWT  $\Longrightarrow$
ALG-$(s,t)$-WT.\end{center}

\vskip 2mm

We say that  multivariate approximation problem suffers from the curse of dimensionality if
there exist positive numbers $C, \vz_0, \alpha$ such that for all
$0 < \vz \le \vz_0$ and infinitely many $d \in \mathbb N$,
$$n^{X,\star}(\vz,d)\ge C(1+\alpha)^d.$$

In the average or worst case setting, we say that the
approximation problem  is exponentially convergent
(EXP) if there exist a number $q \in (0, 1)$ and functions $p, C,
C_1 :\mathbb N \rightarrow (0, \infty)$ such that
\begin{equation}
    e^X(n,d)\le C(d)q^{(n/C_1(d))^{p(d)}}, \ \ \ \ \forall n\in\mathbb{N}_0,\ d\in\mathbb{N}.\label{1.6}
\end{equation}
If \eqref{1.6} holds then
$$p^{*,X}(d) = \sup\{ p(d) : p(d) \ {\rm satisfies}\ \eqref{1.6} \}$$
is called the  exponent of EXP. Moreover, we say that the
approximation problem ${\rm APP}$ or $S$ is uniformly
exponentially convergent (UEXP) if the function $p$ in \eqref{1.6}
can be taken to be a constant, i.e., if $p(d) \equiv p
> 0$ for all $d \in \mathbb N$. If the approximation problem
is UEXP, we let
$$p^{*,X} = \sup\{ p > 0 : p(d) \equiv p\ {\rm satisfies}\ \eqref{1.6}\}$$
denote the exponent of UEXP. For problems in EXP, we can consider
EXP-tractability.

Similar to ALG-tractability we now give the basic notions of
EXP-tractability, mirroring the notions  of ALG-tractability
mentioned above.

Let ${\rm APP}= \{{\rm
    APP}_d\}_{d\in\Bbb N}$, $S= \{S_d\}_{d\in\Bbb N}$, $X\in\{\rm {wor},\,\rm{avg}\}$, and
$\star\in\{{\rm ABS,\,NOR}\}$. In the average or worst  case
setting for error criterion $\star$, we say that ${\rm APP}$ or
$S$ is

$\bullet$ Exponential strongly polynomially tractable (EXP-SPT) if
there exist $ C>0$ and non-negative number $p$ such that
\begin{equation*}n^{X,\,\star}(\va,d)\leq C(\ln\varepsilon^{-1}+1)^{p},\
    \text{for all}\ \varepsilon\in(0,1);\end{equation*}

$\bullet$ Exponential  polynomially tractable (EXP-PT)  if there
exist $C>0$ and non-negative numbers $p,q$ such that
$$n^{X,\,\star}(\va
,d)\leq Cd^{q}(\ln\varepsilon^{-1}+1)^{p},\ \text{for all}\
d\in\mathbb{N},\ \varepsilon\in(0,1);$$

$\bullet$ Exponential  quasi-polynomially tractable (EXP-QPT) if
there exist $C>0$ and non-negative number $t$ such that
\begin{equation*}n^{X,\,\star}(\va,d)\leq C \exp(t(1+\ln{d})(1+\ln(\ln\varepsilon^{-1}+1))),\
    \text{for all}\ d\in\mathbb{N},\
    \varepsilon\in(0,1);\end{equation*}

$\bullet$ Exponential  uniformly weakly tractable (EXP-UWT)  if
$$\lim_{\varepsilon^{-1}+d\rightarrow\infty}\frac{\ln n^{X,\,\star}(\va,d)}{(1+\ln\varepsilon^{-1})^{\sigma}+d^{\tau}}=0,\ \text{for
    all}\ \sigma, \tau>0;$$

$\bullet$ Exponential  weakly tractable (EXP-WT) if
$$\lim_{\varepsilon^{-1}+d\rightarrow\infty}\frac{\ln n^{X,\,\star}(\va,d)}{1+\ln\varepsilon^{-1}+d}=0;$$

$\bullet$ Exponential  $(s,t)$-weakly tractable (EXP-$(s,t)$-WT)
for fixed $s,t>0$ if
$$\lim_{\varepsilon^{-1}+d\rightarrow\infty}\frac{\ln n^{X,\,\star}(\va,d)}{(1+\ln\varepsilon^{-1})^{s}+d^{t}}=0.$$

Clearly,  EXP-$(1,1)$-WT is the same as  EXP-WT. If we have some
notions of EXP-tractability, then we  have the same notion of
ALG-tractability. We  also have for $s,t>0$,

\vskip 2mm
\begin{center} EXP-SPT $\Longrightarrow$  EXP-PT  $\Longrightarrow$   EXP-QPT
$\Longrightarrow$  EXP-UWT  $\Longrightarrow$
 EXP-$(s,t)$-WT.\end{center} \vskip 2mm

    \subsection{Main results}\

In the worst case setting, the authors in \cite{KMU, KP}
considered the approximation problem $S=\{S_d\}_{d\in\Bbb N}$ defined by
    \begin{equation}S_{d}:G^{\alpha,\beta,p}(\Bbb T^{d})\rightarrow L_{2}(\Bbb T^{d})\ \ {\rm with}\ S_df=f.\label{2.1.00}\end{equation}
Among others they obtained sufficient and necessary conditions for
ALG-tractability for the above approximation problem. It follows
from \cite[Theorem 16 (iii)]{KMU} that for $n\ge 2^d$,
 $$e^{\rm wor}(n,d)\asymp
 \exp(-C_{\alpha,p}{\beta}d^{\alpha/p}n^{\alpha/d}),$$where
 $C_{\az,p}$ is a positive constant depending only on $\az,p$.
 This means that $S$ is EXP with the exponent
$$p^{*, \rm wor}(d)=\alpha/d,$$ and is not UEXP.
 Also, the authors gave
results about EXP-tractability in the worst case setting in
\cite[Remark 7.7]{KMU}. We summarize these tractability results as
follows.

\begin{lem} For $\az,\beta,p>0$, consider the approximation problem \eqref{2.1.00} in the worst case
setting.
    For the normalized error criterion or the absolute error criterion in the worst case setting, we have \vskip 2mm

    (i) $S$
    is  not ALG-PT or  ALG-SPT; \vskip 2mm

    (ii) $S$
    is ALG-QPT iff $\alpha\geq p$;\vskip 2mm

    (iii) $S$
    is ALG-UWT, ALG-$(s,t)$-WT for all $s,t>0$, and ALG-WT;\vskip 2mm

(iv) $S$
    is   EXP with the exponent $p^{*, \rm wor}(d)=\alpha/d,$ and is not UEXP; \vskip 2mm

    (v) $S$
    is  not EXP-UWT; \vskip 2mm

(vi) $S$ is EXP-WT if and only if
${\alpha>p}$;\vskip 2mm

    (vii) $S$
    is  EXP-$(s,t)$-WT if and only if $t>1$ or $s>\frac{p}{\az}$.\vskip 2mm

\end{lem}

We now  give main results of this paper. For $\az,\beta,p>0$,
consider the
 approximation problem   ${\rm APP}=\{{\rm
APP}_d\}_{d\in\Bbb N}$ defined over the space $C(\Bbb T^d)$
equipped with a zero-mean Gaussian measure whose covariance kernel
is given as the Gevrey kernel defined by \eqref{2.1.1}. We give
 necessary and sufficient conditions for various notions of ALG-tractability
 in the average case setting. We remark that EXP-tractability in the worst and average case settings have an
intimate connection. Specifically, according to \cite[Theorems 3.2
and 4.2]{X3} and \cite[Theorem 3.2]{PPXD}, we have the same
results in the worst and average case settings using ABS
concerning EXP-WT, EXP-UWT, and EXP-$(s, t)$-WT for $0 < s \le 1$
and $t > 0$. Using the above properties we obtain
 necessary and sufficient conditions for various notions of EXP-tractability
 in the average case setting
   and obtain  complete results about the
tractability. Our main results can be formulated as follows.

\begin{thm} For $\az, \beta,p>0$, consider   the approximation problem \eqref{2.1} in the average case setting with a zero mean Gaussian
measure whose covariance kernel is  the Gevrey kernel given as in
\eqref{2.1.1}.
     For the absolute or normalized error criterion in the average case setting, we
     have \vskip 2mm

  (i) $\rm APP$ is ALG-$(s,t)$-WT for $s>0$ and $t>1$; \vskip 2mm

(ii) $\rm APP$ is ALG-$(s,t)$-WT with $s>0$ and $0<t\leq 1$  if
and only if $\rm APP$  is  ALG-WT if and only if  $\rm APP$ is
ALG-UWT if and only if $\alpha>p$;\vskip 2mm

    (iii) $\rm  APP$ is not ALG-QPT, not ALG-PT or  ALG-SPT;\vskip 2mm

    (iv)  $\rm APP$ suffices from the curse of dimensionality if and only if  $\alpha\leq p$.

\end{thm}

\begin{thm}Consider the approximation problem as in Theorem 2.2.
     For the absolute or normalized error criterion in the average case setting, we
     have \vskip 2mm

    (i) ${\rm APP}$
    is   EXP with the exponent $p^{*, \rm avg}(d)=\alpha/d,$ and is not UEXP; \vskip 2mm

    (ii) $\rm APP$ is not $\rm EXP$-$\rm UWT$,  not EXP-QPT, not EXP-PT or
    EXP-SPT;

   \vskip 2mm

    (iii) $\rm APP$ is $\rm EXP$-$\rm WT$ if and only if
    $\alpha>p$;

   \vskip 2mm

    (iv) $\rm APP$ is EXP-$(s,t)$-WT for $s>0$ and $t>1$;

    \vskip 2mm

    (v) $\rm APP$ is EXP-$(s,t)$-WT with $t\leq 1$ and $0<s\le1$ if and only if
    $s>\frac{p}{\az}$;

    \vskip 2mm

    (vi) $\rm APP$ is EXP-$(s,t)$-WT with $t\leq 1$ and $s>1$ if and only if  $\alpha>p$.
\end{thm}

\begin{rem} We compare the results about ALG-tractability and EXP-tractability in the average and worst  case
settings. In the average and worst  case settings, we never have
ALG-PT, ALG-SPT, EXP-SPT, EXP-PT,  EXP-QPT, and EXP-UWT, and
always have ALG-$(s,t)$-WT and EXP-$(s,t)$-WT with $s>0$ and
$t>1$. We never have ALG-QPT  in the average case setting, whereas
ALG-QPT holds for $\az\ge p$  in the worst case setting.

We always have ALG-UWT and ALG-$(s,t)$-WT with $s>0$ and $0<t\le1$
in the worst case setting, whereas in the average case setting we
have ALG-UWT and ALG-$(s,t)$-WT   with $s>0$ and $0<t\le1$ only
for $\az>p$. Also the sufficient and necessary conditions for
EXP-$(s,t)$-WT with $s>0$ and $0<t\le 1$ in the average case
setting are same or stronger than the ones in the worst case
setting.
\end{rem}
\begin{rem}  Consider the $L_\infty$ approximation problem $S_\infty=\{S_{\infty,d}\}_{d\in\Bbb N}$ in the worst case
setting defined by
    \begin{equation*}S_{\infty,d} :G^{\alpha,\beta,p}(\Bbb T^{d})\rightarrow L_{\infty}(\Bbb T^{d})\ \ {\rm with}\ S_{\infty,d} f=f.\end{equation*}
The worst case error of an algorithm $A_{n,d}$ of the form
\eqref{2.2} is defined by
$$e_\infty^{\rm wor}(A_{n,d})=\sup_{\|f\|_{G^{\alpha,\beta,p}}\le1}
\|f-A_{n,d} f\|_\infty.$$ The $n$-th minimal worst case error is
defined by
$$e_\infty^{\rm{wor}}(n,d)=\inf\limits_{A_{n,d}}e_\infty^{\rm{wor}}(A_{n,d}).$$
Let $\{\lambda_{d,k}\}_{k=1}^{\infty}$ be the nonincreasing
rearrangement of $\{\exp(-2\beta|\mathbf
k|_{p}^{\alpha})\}_{\k\in\mathbb{Z}^{d}}$. It follows from
\cite[Theorem 3.4]{CKS} (or \cite{KWW}) that
\begin{equation}\label{2.1.12}e_\infty^{\rm{wor}}(n,d)=\big(\sum_{k=n+1}^{\infty}\lz_{d,k}\big)^{1/2}.\end{equation}
For $n=0$ we set $A_{0,d}=0$. We obtain the so-called initial
error
$$
e_\infty^{\rm{wor}}(0,d)=\sup_{\|f\|_{G^{\alpha,\beta,p}}\le1}\|f\|_\infty=\big(\sum_{k=1}^{\infty}\lz_{d,k}\big)^{1/2}.$$
 The  information complexity  can be studied using either the
absolute error criterion (ABS) or the normalized error criterion
(NOR). In the worst case setting for $\star\in\{{
    ABS,\,NOR}\}$ we
define the information complexity $n_\infty^{{\rm
wor},\,\star}(\va ,d)$ as
\begin{equation} n_\infty^{ {\rm wor},\, \star}(\va
    ,d):=\inf\{n \ \big|\  e_\infty^{\rm wor}(n,d)\le \vz\, {\rm CRI}_d\}, \label{2.1.10}\end{equation} where
\begin{equation*}
    {\rm CRI}_d:=\left\{\begin{split}
        &\ \ 1, \qquad\,\ \quad\quad\text{for $\star$=ABS,} \\
        &e_\infty^{\rm wor}(0,d), \quad\ \, \text{ for
        $\star$=NOR}.
    \end{split}\right.
\end{equation*}
It follows from \eqref{2.2-0} and \eqref{2.1.12} (or \cite{KWW})
that for $\star\in\{{
    ABS,\,NOR}\}$,
$$e_\infty^{\rm{wor}}(n,d)=e^{\rm{avg}}(n,d),\ \ {\rm and}\ \ n_\infty^{ {\rm wor},\, \star}(\va,d)=n^{ {\rm avg},\,
\star}(\va,d).$$ Hence, the various notions of  tractability of
the $L_\infty$ approximation problem  $S_\infty$ in the worst case
setting are same as the ones of the $L_2$ approximation problem
{\rm APP} in the average case setting. So Theorems 2.2 and 2.3
hold also for the $L_\infty$ approximation problem $S_\infty$. See
\cite[Proposition 5.1]{GWang}.
    \end{rem}

\section{Proofs of Theorems 2.2 and 2.3}

In this section, we give the proofs of Theorems 2.2 and 2.3.
First, we give three auxiliary lemmas.

Let $\ell_{p}^d\ (0< p\le\infty)$ denote the space $\Bbb R^d$
equipped with the $\ell_{p}^d$-norm  defined by
$$\|\x\|_{\ell_{p}^d}:=\bigg\{\begin{array}{ll}\big(\sum_{i=1}^d
|x_i|^p \big)^{\frac 1p},\ \ \ &0< p<\infty;\\ \max_{1\le i\le d}
|x_i|, &p=\infty.\end{array}$$The unit ball of $\ell_{p}^d$ is
denoted by $B\ell_{p}^d$.

Let $A\subset \Bbb R^d$. The covering number $N_\vz(A)$ is the
minimal natural number $n$ such that there are  $n$ points in
$\Bbb R^d$ for which
$$ A\subset \bigcup_{i=1}^n( \x_i+\vz\, B\ell_\infty^d).$$
 The (nondyadic)
entropy numbers are the converse to $N_\vz(A)$ in some sense and
defined by
$$e_n(A,\ell_\infty^d):=\inf\{\vz>0\ |\ N_\vz(A)\le n\}.$$

For  $A=B\ell_p^d,\ 0<p<\infty$, we have  (see \cite{ET, K, RS,
Sc})
 \begin{align*}
&e_n(B\ell_p^d, \ell_\infty^d)\asymp \left\{\begin{matrix}
1, & \ \  1\le n\leq d,\\
 \Big(\frac{\log(1+\frac{d}{\log n})}{\log n}\Big)^{1/p},&\ \  d\le n\le 2^d, \\
 d^{-1/p}n^{-1/d},&\ \  n\ge 2^d,
\end{matrix}\right.
\end{align*}
with the equivalent constants independent of $d$ and $n$. Here,
$\log x=\log_2 x$, and $A\asymp B$ means that there exist two
constant $c$ and $C$ which are called the equivalence constants
such that $c A\le B\le C A$. It follows that for $0< p<\infty$ and
$\vz\in(0,1)$,
\begin{equation}\label{3.0-0}\ln (N_\vz(B\ell_p^d))\asymp
\bigg\{\begin{array}{ll}\vz^{-p}\ln(2d\vz^p), &d\vz^p\ge 1,\\ d\ln
(2(d\vz^{p})^{-1}), &d\vz^p\le1,
\end{array}\end{equation}where the equivalent constants depend only on $p$,
but are independent of $d$ and $\vz$.

 For $A\subset \Bbb R^d$, let $G(A)$ be the grid number of points in $A$ that lie on the grid $\Bbb Z^d$, i.e.,
$$G(A)=\#(A\cap \Bbb Z^d),$$where $\#A$ denotes the number of elements in a set
$A$. We set $$A=\{\k\in\Bbb Z^{d}\,\big|\,
|\k|_p^p=\sum\limits_{i=1}^{d}|k_{i}|^{p}\le m,m\in\Bbb N\big\}.$$ It follows
from \cite[Lemma 5]{KMU} that
$$N_1(A)\le G(A)\le N_{1/4}(A).$$
Using \eqref{3.0-0} and the same method in \cite[Lemma 3.1]{W}, we
obtain the following lemma.

\begin{lem}\label{l3.4.1}
We have for $m,d\geq 2$,
\begin{equation}
\ln{\Big(\#\big\{\k\in\Bbb Z^{d}\,\big|\,
|\k|_p^p=\sum\limits_{i=1}^{d}|k_{i}|^{p}\le
m\big\}\Big)}\asymp\bigg\{
\begin{array}{rcl}
d\ln{\left(\frac{2m}{d}\right)}, & & d\leq m,\\
m\ln{\left(\frac{2d}{m}\right)}, & & d\geq m,
\end{array}\nonumber
\end{equation}where the equivalent constants are independent of $d$ and $n$.
\end{lem}
In the average case setting, a sufficient condition for ALG-WT was
given in \cite[Lemma 6.5]{NW1}. Using the same induction, we
obtain  a similar sufficient condition for ALG-$(s,t)$-WT  in the
average case setting as follows.
\begin{lem}\label{l3.4.2}
If there exists a constant $\tau\in(0,1)$ such that
$$\lim_{d\rightarrow\infty}{\frac{\ln{\Big(\Big(\sum\limits_{j=1}^{\infty}\lambda_{d,j}
^{\tau}\Big)^{\frac{1}{\tau}}/CRI_{d}\Big)}}{d^{t}}}=0,$$ then
ALG-$(s,t)$-weak tractability holds for this $t>0$ and any $s>0$.
\end{lem}

We set $$\Lz_d=(e^{\rm
avg}(0,d))^2=\sum\limits_{k=1}^{\infty}\lz_{d,k}.$$

\begin{lem}\label{l3.4.3}
For $\vz\in(0,1)$ we have
\begin{equation}n^{\rm NOR}(\vz,d)\ge(1-\vz^{2})\Lz_{d}.\label{3.1}\end{equation}
\end{lem}
\begin{proof}
Let $n_0=n^{\rm NOR}(\vz,d)$, by \eqref{2.1.10} and \eqref{2.2-0}
we get
$$\sum\limits_{k=n_0+1}^{\infty}\lz_{d,k}\le\vz^2
\sum\limits_{k=1}^{\infty}\lz_{d,k},$$
which implies that
$$\Lz_d-\sum\limits_{k=1}^{n_0}\lz_{d,k}\le\vz^2\Lz_d.$$
Noting that $\lz_{d,1}=1$, we have
$$n_0=n_0\lz_{d,1}\ge\sum\limits_{k=1}^{n_0}\lz_{d,k}\ge(1-\vz^2)\Lz_d.$$This
completes the proof of Lemma 3.3.
\end{proof}

Now we prove Theorems 2.2 and 2.3.

\

\noindent{\it Proof of Theorem 2.2.}\

(i) If $t>1$, then  we shall show that
\begin{equation}\lim_{d\rightarrow\infty}\frac{\ln\Big(\sum\limits_{k=1}^{\infty}\lz_{d,k}^{\frac
1 2}\Big)^2} {d^t}=0.\label{3.3}\end{equation}If \eqref{3.3} is
proved, then by Lemma 3.2 and \eqref{2.33} we get that APP is
ALG-$(s,t)$-WT with $s>0$ and $t>1$ for the absolute or normalized
error criterion. This completes the proof of (i).\vskip 2mm

It remains to prove \eqref{3.3}. For  $\alpha\ge p$, we have
\begin{align*}\sum_{k=1}^{\infty}\lz_{d,k}^{\frac 1
2}&=\sum_{\mathbf{k}\in\mathbb{Z}^d}
\exp(-\beta|\mathbf{k}|_p^\alpha)\le
\sum_{\mathbf{k}\in\mathbb{Z}^d} \exp(-\beta|\mathbf{k}|_p^p)\\
&=\sum_{\mathbf{k}=(k_1,\ldots,k_d)\in\mathbb{Z}^d}
\exp\Big(-\beta\sum_{i=1}^{d}|k_i|^p\Big)\nonumber\\
&=\sum_{\mathbf{k}=(k_1,\ldots,k_d)\in\mathbb{Z}^d}\prod_{i=1}^{d}
\exp\big(-\beta|k_i|^p\big)\nonumber\\
&=(1+2\sum_{h=1}^{\infty} \exp(-\beta h^p))^d=:A^d,\end{align*}
where $$A=1+2\sum_{h=1}^{\infty} \exp(-\beta h^p)>1$$ is a
constant independent of $d$. It follows that for $t>1$,
$$0\le \lim_{d\rightarrow\infty}\frac{\ln\big(\sum\limits_{k=1}^{\infty}\lz_{d,k}^{\frac 1 2}\big)^2}{d^t}
\le \lim_{d\rightarrow\infty}\frac{2\ln A}{d^{t-1}}=0,$$proving
\eqref{3.3}.

For $\alpha<p$ we have
 \begin{align}
\sum_{k=1}^{\infty}\lz_{d,k}^{\frac 1
2}&=\sum_{\mathbf{k}\in\mathbb{Z}^d}
\exp(-\beta|\mathbf{k}|_p^\alpha)=\sum_{m=1}^{\infty}\sum_{m-1\le|\mathbf{k}|_p^p<m}
\exp(-\beta|\mathbf{k}|_p^\alpha)\nonumber\\
&\le\sum_{m=1}^{\infty}\exp(-\beta(m-1)^\frac{\alpha}{p})
\sum_{m-1\le|\mathbf{k}|_p^p<m}1\nonumber\\
&\le\sum_{m=1}^{\infty}\exp(-\beta(m-1)^\frac{\alpha}{p})\#\{\k\in \Bbb Z^d\,\big|\ |\mathbf{k}|_p^p<m\}\nonumber\\
&\le\sum_{m=1}^{d}\exp(-\beta(m-1)^\frac{\alpha}{p})\exp(C_1m\ln\frac{2d}{m})\nonumber\\
&\quad\ \ \ +\sum_{m=d+1}^{\infty}\exp(-\beta(m-1)^\frac{\alpha}{p})\exp(C_1d\ln\frac{2m}{d})\nonumber\\
&=:I_1+I_2,\label{3.4}
\end{align}where $C_1$ is a positive constant independent of
$d$.

 We get for $\az<p$
\begin{align}
I_1&:=\sum_{m=1}^{d}\exp\big(-\beta(m-1)^{\frac{\alpha}{p}}\big)
\exp\big(C_1 m\ln(\frac{2d}{m})\big)\nonumber\\
&\leq\exp(C_1 d\ln (2d))\sum_{m=1}^{d}\exp\big(-\beta(m-1)^{\frac{\alpha}{p}}\big)\nonumber\\
&\le N_1\exp(C_1 d\ln (2d)),\label{3.5}
\end{align}where $$N_1:=\sum_{m=1}^{\infty}\exp\big(-\beta(m-1)^{\frac{\alpha}{p}}\big)<+\infty$$ is a constant. Since $$\lim_{m\rightarrow\infty}\frac{C_1\ln (2m)}
{\frac{\beta}{2}(m-1)^{\frac {\alpha} {2p}}}=0,$$ there exists a
positive integer $M>0$ such that for any  $m>M$, we have
\begin{equation}
C_1\ln (2m)\leq\frac{\beta}{2}(m-1)^{\frac {\alpha}
{2p}}.\label{3.6}
\end{equation}
We choose $m_1=\max\{M,2+\lfloor d^{\frac{2p}{\alpha}}\rfloor\}$.
For $m>m_1$ we  obtain
\begin{equation}
m_1\geq d+1\ \text{and}\ m-1\geq
 d^{\frac{2p}{\alpha}}.\label{3.7}
\end{equation}It follows from \eqref{3.6} and \eqref{3.7} that
\begin{equation}\frac{\beta}{2}(m-1)^{\frac {\alpha} {p}}=\frac{\beta}{2}(m-1)^{\frac {\alpha} {2p}} (m-1)^{\frac {\alpha} {2p}}\geq C_1\ln(2m) d.\label{3.8}
\end{equation}For sufficiently large $d$, we have
$$m_1=2+\lfloor d^{\frac{2p}{\alpha}}
\rfloor\le 2+d^{\frac{2p}{\alpha}} \leq
2^{\frac{2p}{\alpha}-1}d^{\frac{2p}{\alpha}},$$ which yields
\begin{equation}\ln(2m_1)\leq\frac{2p}{\alpha}\ln(2d).\label{3.9}
\end{equation}
We have
\begin{align}
I_2&:=\sum_{m=d+1}^{\infty}\exp(-\beta(m-1)^\frac{\alpha}{p})\exp(C_1d\ln\frac{2m}{d})\notag\\
&\le
\Big(\sum_{m=d+1}^{m_1}+\sum_{m=m_1+1}^{\infty}\Big)\exp\big(-\beta(m-1)^{\frac{\alpha}{p}}\big)
\exp\big(C_1 d\ln(2m)\big)\notag\\
&=:I_{23}+I_{24}.\label{3.9-1}
\end{align}
It follows from \eqref{3.8} and \eqref{3.9} that
\begin{align}
I_{24}&:=\sum_{m=m_1+1}^{\infty}\exp\big(-\beta(m-1)
^{\frac{\alpha}{p}}\big) \exp\big(C_1 d\ln(2m)\big)\notag\\
&\leq\sum_{m=1}^{\infty}\exp\big(-\frac{\beta}{2}(m-1)
^{\frac{\alpha}{p}}\big)=: N_2<+\infty,\label{3.10}
\end{align}
and
\begin{align}
I_{23}&:=\sum_{m=d+1}^{m_1}\exp\big(-\beta(m-1)^{\frac{\alpha}{p}}\big)
\exp\big(C_1 d\ln(2m)\big)\notag\\ &\leq\exp(C_1
d\ln(2m_1))\sum_{m=d+1}^{m_1}\exp\big(-{\beta}(m-1)
^{\frac{\alpha}{p}}\big)\notag\\ &\le
N_1\exp\big(\widetilde{C}d\ln(2d)\big),\label{3.11}
\end{align}where $\widetilde{C}=C_1\frac{2p}{\alpha}$.
By \eqref{3.4}, \eqref{3.5}, \eqref{3.9-1},  \eqref{3.10}, and
\eqref{3.11} we obtain for $\alpha<p$,
\begin{align*}
\sum_{k=1}^{\infty}\lz_{d,k}^{\frac 1 2}&\leq N_1\exp\big(C_1
d\ln(2d)\big)+N_2+
N_1\exp\big(\widetilde{C}d\ln(2d)\big)\\
&\leq\max\{2N_1,N_2\}\exp\big(\max\{C_1,\widetilde{C}\}d\ln(2d)\big).
\end{align*}
Hence, for $t>1$  we have
$$\lim_{d\rightarrow\infty}\frac{\ln\Big(\sum\limits_{k=1}^{\infty}\lz_{d,k}^{\frac 1 2}\Big)^2}
{d^t}\leq\lim_{d\rightarrow\infty}\frac{2\ln\max\{2N_1,N_2\}+
2\max\{C_1,\widetilde{C}\}d\ln(2d)} {d^t}=0,$$proving
\eqref{3.3}.\vskip 2mm

(ii) It suffices to show that if $\az> p$, APP is ALG-UWT for ABS
or NOR, and if $\az\le p$, APP is not ALG-WT or ALG-$(s,t)$-WT
with $s>0$ and $0<t\le 1$.

 First we show that  if $\az> p$, APP
is ALG-UWT for ABS or NOR.  By Lemma 3.2 and \eqref{2.33} it
suffices to prove \eqref{3.3} for any $t>0$.

We recall that$$\sum_{k=1}^{\infty}\lz_{d,k}^{\frac 1 2}\le
I_1+I_2,$$and
\begin{align}
I_2&:=\sum_{m=d+1}^{\infty}\exp(-\beta(m-1)^\frac{\alpha}{p})\exp(C_1d\ln\frac{2m}{d})\notag\\
&\le\sum_{m=d+1}^{\infty}\exp(-\beta(m-1)^\frac{\alpha}{p})\exp(C_1m\ln(2m)).\notag\end{align}
Since for  $\alpha>p$,
$$\lim_{m\rightarrow+\infty}\frac{C_1m\ln(2m)}{\frac\beta 2
 (m-1)^\frac{\alpha}{p}}=0,$$we have for sufficiently large $m$,
$$C_1m\ln(2m)\le \frac\beta 2
 (m-1)^\frac{\alpha}{p}.$$
It follows that for sufficiently large $d$,
\begin{equation}\label{3.13}I_2\le \sum_{m=d+1}^{\infty}\exp(-\frac{\beta}2(m-1)^\frac{\alpha}{p})
\le
\sum_{m=1}^{\infty}\exp(-\frac{\beta}2(m-1)^\frac{\alpha}{p})=:N_2<+\infty.\end{equation}

Next, we estimate $I_1$. Choose
$m_0=\lfloor(\frac{4C_1}{\beta}\ln(2d))^{\frac{p}{\az-p}}\rfloor+1$.
We recall that \begin{align*}
I_1&:=\sum_{m=1}^{d}\exp(-\beta(m-1)^\frac{\alpha}{p})\exp\big(C_1m\ln\frac{2d}m\big)\\
&\le \Big(\sum_{m=1}^{m_0}+\sum_{m=m_0+1}^{d}\Big)\exp(-\beta(m-1)^\frac{\alpha}{p})\exp\big(C_1m\ln(2d)\big)\\
&=:I_{11}+I_{12}.
\end{align*} If $m>m_0$, then we have for sufficiently large
$d$,
$$m-1\ge m_0\ge
(\frac{4C_1}{\beta}\ln(2d))^{\frac{p}{\az-p}},\ \ \ 4C_1\ln(2d)\le
\beta (m-1)^{\frac \az p-1},$$ and thence,
$$C_1m\ln(2d)\le (m-1) \,2C_1\ln(2d)\le  \frac\beta 2(m-1)^{\frac \az
p}.$$ We also have for sufficiently large $d$,
$$C_1m_0\ln(2d)\le C_2(\ln(2d))^\frac{\az}{\az-p}, \ \ \ C_2=C_1\big(\frac{4C_1}{\beta}\big)^
{\frac{p}{\az-p}}+C_1.$$ It follows that for sufficiently large
$d$,
\begin{align}
I_{11}&:=\sum_{m=1}^{m_0}\exp(-\beta(m-1)^\frac{\alpha}{p})\exp(C_1m\ln(2d))\nonumber\\
&\le\exp(C_1m_0\ln2d)\sum_{m=1}^{\infty}\exp\big(-{\beta}(m-1)^\frac{\alpha}{p}\big)\nonumber\\
&\le
N_1\exp\big(C_2(\ln(2d))^{\frac{\alpha}{\az-p}}\big),\label{3.14}
\end{align}
and \begin{align}
I_{12}&:=\sum_{m=m_0+1}^{d}\exp(-\beta(m-1)^\frac{\alpha}{p})\exp\big(C_1m\ln(2d)\big)\notag\\
&\le\sum_{m=m_0+1}^{d}\exp(-\frac{\beta}{2}(m-1)^\frac{\alpha}{p})\nonumber\\
&\le\sum_{m=1}^{\infty}\exp(-\frac{\beta}{2}(m-1)^\frac{\alpha}{p})
=N_2<+\infty.\label{3.15}
\end{align}

By \eqref{3.4}, \eqref{3.13}, \eqref{3.14}, and \eqref{3.15}, we
 obtain
\begin{align*}
\sum_{k=1}^{\infty}\lz_{d,k}^{\frac 1 2}&\le I_{11}+I_{12}+I_2\le
2N_2+ N_1\exp\big(C_2(\ln(2d))^{\frac{\alpha}{\az-p}}\big)\\ &\le
(2N_2+ N_1) \exp\big(C_2(\ln(2d))^{\frac{\alpha}{\az-p}}\big).
\end{align*}It follows that for any $t>0$,
$$\lim_{d\rightarrow\infty}\frac
{\ln\Big(\sum\limits_{k=1}^{\infty}\lz_{d,k}^{\frac 1
2}\Big)^2}{d^t} \le
\lim_{d\rightarrow\infty}\frac{2\ln(2N_2+N_1)+2C_2(\ln(2d))^{\frac{\alpha}{\az-p}}}{d^t}=0,$$
proving \eqref{3.3}.\vskip1mm

Next we show that if $\az\le p$, APP is not ALG-WT or
ALG-$(s,t)$-WT with $s>0$ and $0<t\le 1$ for ABS or NOR.  By
\eqref{2.33} it suffices to show it for NOR.  For $\az\le p$, we
have
\begin{align*}
\Lz_d&:=\sum_{k=1}^{\infty}\lz_{d,k}=\sum_{\mathbf{k}\in\mathbb{Z}^d}
\exp(-2\beta|\mathbf{k}|_p^\alpha)\ge
\sum_{\mathbf{k}\in\mathbb{Z}^d}
\exp(-2\beta|\mathbf{k}|_p^p)\notag\\&=(1+2\sum_{h=1}^{\infty}
\exp(-2\beta h^p))^d=: A^d,
\end{align*}where $A=1+2\sum_{h=1}^{\infty} \exp(-2\beta h^p)>1$ is a constant.
By \eqref{2.33} and \eqref{3.1}, for $\vz\in (0,1/2)$ we get
\begin{equation}n^{\rm{ABS}}(\vz,d)\ge n^{\rm{NOR}}(\vz,d)\ge
n^{\rm{NOR}}(1/2,d)\ge\frac{3}{4}A^d.\label{3.34}\end{equation}
Hence for any $s>0$ and $0<t\le1 $ we have
$$\lim_{(1/2)^{-1}+d\to\infty}\frac{\ln
(n^{\rm{NOR}}(1/2,d))}{(1/2)^{-s}+d^t}\ge \lim_{d\to\infty}
\frac{\ln 3/4+d\ln A}{2^s+d^t}\ge \ln A>0,$$ and
$$\lim_{(1/2)^{-1}+d\to\infty}\frac{\ln
(n^{\rm{NOR}}(1/2,d))}{(1/2)^{-1}+d}\ge \lim_{d\to\infty}
\frac{\ln 3/4+d\ln A}{2+d}=\ln A>0.$$ This  means that  APP is not
ALG-$(s,t)$-WT with $s>0$ and $0<t\le1$ or ALG-WT for NOR.

This completes the proof of (ii).\vskip 2mm

(iii) By (ii) we know that for $\alpha\le p$, $\rm{APP}$ is not
ALG-UWT for ABS or NOR, then $\rm{APP}$ is not ALG-QPT for ABS or
NOR. Hence, it suffices to show that for $\az>p$, $\rm{APP}$ is
not ALG-QPT for NOR.

For $\az>p$,   by \eqref{3.1} we have
 \begin{equation}
n^{\rm{NOR}}\big(\frac1 2,d\big)\ge\frac3 4\Lz_d.\label{3.16}
\end{equation}
By Lemma 3.1, there exists a constant $C_0>0$ such that for
$m_0\le d$,
$$\#\{\k\in\Bbb Z^d\,|\,|\k|_p^p\le m_0\}\ge
\exp(C_0m_0\ln\frac{2d}{m_0}).
$$ Choose
$m_0=\lfloor\big(\frac{C_0}{12\beta}\ln(2d)\big)^{\frac{p}{\az-p}}\rfloor$.
Then for sufficiently large $d$,  we have $m_0\le d$, and $$\ln
m_0\le\frac{p}{\az-p}\Big(\ln\frac{C_0}{12\beta}+\ln\ln2d\Big)\le\frac1
2\ln2d.$$ Hence we obtain for sufficiently large $d$,
\begin{align*} C_0m_0\ln\frac{2d}{m_0}&\ge \frac{C_0}2m_0\ln (2d)\ge
\frac{C_0}4 \big(\frac{C_0}{12\beta}\ln(2d)\big)^{\frac{p}{\az-p}}
\ln(2d)\\ &=3\beta
\big(\frac{C_0}{12\beta}\ln(2d)\big)^{\frac{\az}{\az-p}}\ge 3\beta
m_0^{\az/p}.
\end{align*}
 It
follows that
\begin{align}
\Lz_d&=\sum_{\mathbf{k}\in\mathbb{Z}^d}
\exp(-2\beta|\mathbf{k}|_p^\alpha)\nonumber\\
&\ge\sum_{|\mathbf{k}|_p^p\le m_0}
\exp(-2\beta|\mathbf{k}|_p^\alpha)\nonumber\\
&\ge\exp(-2\beta m_0^\frac{\alpha}{p})\sum_{|\mathbf{k}|_p^p\le m_0}1\nonumber\\
&\ge\exp(-2\beta
m_0^\frac{\alpha}{p})\exp\big(C_0m_0\ln\big(\frac{2d}{m_0}\big)\big)\notag\\
&\ge \exp(\beta m_0^\frac{\alpha}{p}).\label{3.17}
\end{align}
By \eqref{3.17} we  have for $\az>p$,
$$\lim_{d\rightarrow+\infty}\frac{\ln n^{\rm{NOR}}\big(\frac1 2,d\big)}{(1+\ln2)(1+\ln d)}
\ge\lim_{d\rightarrow+\infty}\frac{\ln \frac34+\beta
m_0^\frac{\alpha}{p}}{(1+\ln2)(1+\ln d)} =+\infty,$$ which
contradicts the definition of ALG-QPT. Hence, $\rm{APP}$ is not
ALG-QPT for NOR.

This completes the proof of (iii). \vskip2mm

(iv) If $\az> p$, from (ii) we know that $\rm{APP}$ is  ALG-UWT
for ABS or NOR. Hence $\rm{APP}$ does not suffer from the curse of
dimensionality for ABS or NOR. If $\az\le p$, then by \eqref{3.34}
we obtain $\rm{APP}$ suffers from the curse of dimensionality for
ABS or NOR. This completes the proof of (iv). \vskip1mm

The proof of Theorem 2.2 is finished. $\hfill\Box$

\

\noindent{\it Proof of Theorem 2.3. }

(i) In the worst case setting, it follows from Lemma 2.1 that $S$
is EXP with exponent $p^{*,\rm wor}(d)=\az/d$ and not UEXP. Using
the same method as in the proof of \cite[Theorem 4.1]{LX2} we
obtain that in the average case setting APP is EXP with exponent
$p^{*,\rm avg}(d)=\az/d$ and is not UEXP.\vskip2mm

(iii-v) According to \cite[Theorems 3.2 and 4.2]{X3} and
\cite[Theorem 3.2]{PPXD}, we have the same results in the worst
and average case settings using ABS concerning EXP-WT,  EXP-$(s,
t)$-WT with $0 < s \leq 1$ and $t > 0$. It follows from the proof
of \cite[Theorem 1.2]{W} that in the average case setting, APP is
EXP-$(s, t)$-WT with $0 < s \leq 1$ and $t > 0$ for ABS if and
only  if APP is EXP-$(s, t)$-WT with $0 < s \leq 1$ and $t > 0$
for NOR.\vskip1mm

Applying the above results to Lemma 2.1 (vi), (vii), we obtain
that in the average case setting for ABS or NOR, APP is EXP-WT if
and only if $\az>p$,  $\rm APP$ is EXP-$(s,t)$-WT with $t\leq 1$
and $0<s\le1$ if and only if  $s>\frac{p}{\az}$. This completes
the proof of (iii) and (v).\vskip1mm

We also have in the average case setting and for ABS or NOR,  $\rm
APP$ is EXP-$(s,t)$-WT with $t>1$ and $0<s\le1$. It follows that
$\rm APP$ is EXP-$(s,t)$-WT with $t>1$ and $s>1$. Hence,  $\rm
APP$ is EXP-$(s,t)$-WT with $t>1$ and  $s>0$. This completes of
proof of (iv).\vskip2mm

(ii) It follows from (v) that in the average case setting  for ABS
or NOR,  $\rm APP$ is not EXP-UWT, and hence not EXP-QPT, not
EXP-PT or EXP-SPT. This completes of proof of (ii).\vskip2mm

(vi) It follows from (v) that in the average case setting  for ABS
or NOR,  $\rm APP$ is EXP-$(1,t)$-WT with $t\leq 1$ if and only if
$\az>p$. Hence, if $\az>p$, then for ABS or NOR, $\rm APP$ is
EXP-$(s,t)$-WT with $t\leq 1$ and $s> 1$. On the other hand, if
$\rm APP$ is EXP-$(s,t)$-WT with $t\leq 1$ and $s> 1$ for ABS or
NOR, then $\rm APP$ is ALG-$(s,t)$-WT with $t\leq 1$ and $s> 1$
for ABS or NOR. By Theorem 2.2 (ii) we obtain that $\az>p$. This
completes of proof of (vi).\vskip1mm

The proof of Theorem 2.3 is finished. $\hfill\Box$

\

\noindent{\bf Acknowledgment}  This work was supported by the
National Natural Science Foundation of China (Project no.
12371098).


\begin{thebibliography}{99}
\bibitem{CW1}J. Chen, H. Wang, Average Case tractability of  multivariate  approximation with
Gaussian  kernels, J. Approx. Theory 239 (2019) 51-71.

\bibitem{CW2}J. Chen, H. Wang, Approximation numbers of Sobolev and Gevrey
type embeddings on the sphere and on the sphere and on the ball -
Preasymptotics, asymptotics, and tractabilty, J. Complexity 50
(2019) 1-24.

\bibitem{CKS} F. Cobos, T. K\"uhn, W. Sickel,  Optimal approximation of multivariate periodic Sobolev functions in the
sup-norm, J. Funct. Anal. 270(11) (2016) 4196-4212.

\bibitem{DKPW}J. Dick, P. Kritzer, F. Pillichshammer, H. Wo\'zniakowski,
Approximation of analytic functions in Korobov spaces, J.
Complexity 30 (2014) 2-28.

\bibitem{DLPW} J. Dick, G. Larcher, F. Pillichshammer, H. Wo\'zniakowski,
 Exponential convergence and tractability of multivariate integration for Korobov spaces, Math. Comp. 80 (2011) 905-930.

\bibitem{ET}D. E. Edmunds,  H. Triebel, Function Spaces, Entropy Numbers,
Differential Operators, Cambridge Tracts in Math. 120, Cambridge
University Press, Cambridge, UK, 2008.

\bibitem{GWang} J.  Geng, H. Wang,  On the power of standard
information for tractability for $L_\infty$ approximation of
periodic functions in the worst case setting, J. Complexity 80
(2024) 101790, 25 pp.

\bibitem{G} M. Gevrey, Sur la nature analytique des solutions des
\'equations aux d\'eriv\'ees partielles. premier m\'emoire, Ann.
Sci. \'Ec. Norm. Sup\'er., 35 (1918) 129-190.

\bibitem{GW} M. Gnewuch,  H. Wo\'zniakowski, Quasi-polynomial tractability,
J. Complexity 27 (2011) 312-330.

\bibitem{IKPW}C. Irrgeher, P. Kritzer, F. Pillichshammer, H. Wo\'zniakowski,
 Tractability of multivariate approximation defined over Hilbert spaces with exponential weights, J. Approx. Theory 207 (2016) 301-338.

 \bibitem{K} T. K\"uhn, A lower estimate for entropy numbers, J. Approx. Theory, 110 (2001)
120-124.

\bibitem{KMU} T. K\"uhn, S. Mayer, T. Ullrich, Counting via entropy: new preasymptotics for the approximation numbers of Sobolev
embeddings, SIAM J. Numer. Anal. 54 (6) (2016) 3625-3647.

\bibitem{KP} T. K\"uhn,  M. Petersen,  Approximation in periodic
Gevrey spaces, J. Complexity 73 (2022) 101665, 24 pp.

\bibitem{KW}P. Kritzer, H. Wo\'zniakowski, Simple characterizations of exponential tractability for linear multivariate problems, J. Complexity 51 (2019) 110-128.

\bibitem{KWW} F. Y. Kuo, G. W. Wasilkowski, H.  Wo\'niakowski,
 Multivariate $L_\infty$ approximation in the worst case setting
over reproducing kernel Hilbert spaces. J. Approx. Theory 152 (2)
(2008) 135-160.

\bibitem{LX2} Y. Liu, G. Xu, Average case tractability of a multivariate approximation problem, J. Comlexity 43 (2017) 76-102.



\bibitem{LW} W. Lu, H. Wang, On the power of standard information for tractability for
$L_2$-approximation in the average case setting, J. Complexity 70
(2022) 101618, 22 pp.

\bibitem{NW1} E. Novak, H. Wo\'zniakowski, Tractablity  of Multivariate Problems, Volume I: Linear Information, EMS, Z\"urich, 2008.

\bibitem{NW2} E. Novak, H. Wo\'zniakowski, Tractablity  of Multivariate Problems, Volume II: Standard Information for Functionals, EMS, Z\"urich, 2010.

\bibitem{NW3} E. Novak, H. Wo\'zniakowski, Tractablity  of Multivariate Problems, Volume III: Standard Information for Operators, EMS, Z\"urich, 2012.

\bibitem{PP} A. Papageorgiou, I. Petras, A new criterion for tractability of multivariate problems, J. Complexity 30 (2014) 604-619.

\bibitem{PPW}A. Papageorgiou, I. Petras, H. Wo\'zniakowski, $(s, \ln^\kappa)$-weak
tractability of linear problems, J. Complexity 40 (2017) 1-16.

\bibitem{PPXD} A. Papageorgiou, I. Petras, G. Q. Xu, D. Yanqi, EC-$(s,t)-$weak tractability of multivariate linear problems in the average case setting,
J. Complexity 55 (2019) 101425, 26 pp.

\bibitem{RS} C. Richter and M. Stehling, Entropy numbers and lattice
arrangements in $\ell_\infty(\Gamma)$, Math. Nachr., 284 (2011)
818-830.

\bibitem{R} L. Rodino, Linear Partial Differential Operators in Gevrey Spaces,
World Scientific, 1993.

\bibitem{Sc} C. Sch\"utt, Entropy numbers of diagonal operators between
symmetric Banach spaces, J. Approx. Theory, 40 (1984) 121-128.

\bibitem{S} P.  Siedlecki, Uniform weak tractability,  J. Comlexity 29(6) (2013) 438-453.

\bibitem{SW15} P. Siedlecki, M. Weimar, Notes on $(s,t)$-weak tractability: a refined classification of problems with (sub)exponential information
complexity, J. Approx. Theory 200 (2015) 227-258.


\bibitem{SW} I. H. Sloan, H. Wo\'zniakowski, Multivariate approximation for analytic functions with Gaussian kernels, J. Complexity 45 (2018) 1-21.


\bibitem{TWW}J. F. Traub, G. W. Wasilkowski, H. Wo\'zniakowski.
Information-Based Complexity. Academic Press, New York, 1988.

\bibitem{W}H. Wang, A note about EC-$(s, t)$-weak tractability of
multivariate approximation with analytic Korobov kernels, J.
Complexity 55 (2019) 101412, 19pp.
\bibitem{X3}G. Xu, Exponential convergence-tractability of general linear problems in the average case setting, J. Complexity 31 (2015)
617-636.

\bibitem{X5}G. Xu, On the power of standard information for
$L_2$-approximation in the average case setting, J. Complexity 59
(2020) 101482, 20 pp.
\end{thebibliography}
\end{document}